\RequirePackage{fix-cm}
\documentclass[smallextended]{svjour3}       
\smartqed  
\usepackage{ulem}
\usepackage{xcolor}
\usepackage{graphicx}
\usepackage{amsmath}
\usepackage{graphicx,amssymb}

 \usepackage{mathptmx}
 \usepackage[ruled,vlined]{algorithm2e}

\newtheorem{exmp}{Example}

\begin{document}

\title{Total Positivity of Analytic Bases through Symmetric Functions
\protect\thanks{This work was partially supported by Spanish research grants PID2022-138569NB-I00 (MCI/AEI) and RED2022-134176-T (MCI/AEI) and by Gobierno de Aragón (E41$\_$23R, S60$\_$23R).}}

\subtitle{}

\author{P. D\'iaz \and E. Mainar}

\authorrunning{Short form of author list} 

\institute{P. D\'iaz   \at
          \\Departamento de Matem\'{a}tica Aplicada/IUMA, Universidad de Zaragoza, Spain \email{pablodiaz@unizar.es}   
          \and
          E. Mainar \at
             \\Departamento de Matem\'{a}tica Aplicada/IUMA, Universidad de Zaragoza, Spain \email{esmemain@unizar.es} } 

\date{Received: date / Accepted: date}

\maketitle

\begin{abstract} 
  This paper studies the bidiagonal factorization of the collocation matrices of analytic bases using symmetric functions. Explicit formulas for their initial minors are derived in terms of Schur functions. The structure of these formulas permits establishing sufficient conditions for the total positivity of generic systems of analytic functions. In addition, they have been found to lead to generalizations of the Cauchy identity for certain families of functions.


\end{abstract}

\section{Introduction}
A class of matrices that has been extensively analyzed during the last decades are those with all their minors being nonnegative---these are known as totally positive (TP) matrices. The set of TP matrices possesses rich mathematical properties and numerous applications, attracting significant interest across various fields of mathematics, including approximation theory, combinatorics, computer-aided geometric design, and~economics (cf.~\cite{Ando,FJ11,Gasca1,Gasca2,Gasca3}).

An important property of TP matrices is their closure under multiplication, meaning that the product of two TP matrices is also TP (see Theorem 3.1 of~\cite{Ando}).  This property opened the possibility of factorizing TP matrices into products of simpler TP factors, which is a topic widely studied in numerical linear algebra and computational mathematics, providing computational advantages, particularly in numerical algorithms where efficiency and stability are key concerns  \cite{Marco3,dop1,dop2,Koev1,Koev2,Koev4,Demmel1,Demmel2}.

As shown in \cite{Gasca2,Gasca3}, any nonsingular  TP matrix \( A \in \mathbb{R}^{n \times n} \) admits a bidiagonal factorization of the form  
\begin{equation}\label{BDAfac}
    A = F_{n-1} \cdots F_1 D G_1 \cdots G_{n-1},
\end{equation}  
where \( F_i, G_i \in \mathbb{R}^{n \times n} \) are TP lower and upper triangular bidiagonal matrices, and \( D \) is a nonsingular TP diagonal matrix.  
This factorization enables the highly accurate solution of various numerical linear algebra problems \cite{Martinez2024}.   When certain conditions are met, machine precision is achieved even for extremely ill-conditioned problems and large dimensions—a property known as High Relative Accuracy (HRA).

It turns out that all the parameters of the factorization \eqref{BDAfac} can be explicitly determined from the initial minors of $A$ \cite{Gasca1,Gasca2}, 
\begin{equation}\label{Aminors}
    \det A[i-j+1, \dots, i \mid 1, \dots, j], \quad  
    \det A^T[i-j+1, \dots, i \mid 1, \dots, j], \quad  
     \quad 0 < j \leq i \leq n.
\end{equation}  
 Moreover, it is well known that a matrix \( A \) is TP if and only if all the initial minors in \eqref{Aminors} are nonnegative \cite{Gasca1,Gasca3}.

 Collocation matrices encode the values of a given system of functions at a sequence of parameters within the domain (see \eqref{eq:ColMF}). 
From an algebraic point of view, TP colocation matrices play a fundamental role in discretization, interpolation, and various other applications. 
A crucial observation is that any minor of a collocation matrix is an antisymmetric function of the subset of nodes involved. 
 This property extends to their initial minors \eqref{Aminors}. Thus, for TP collocation matrices, all entries in the factorization \eqref{BDAfac}  inherit a natural symmetry, a fact that lies at the core of an interesting connection between the bases whose collocation matrices are TP and symmetric functions.

The relationship above was first highlighted in~\cite{YHZ} in the particular context of  Cauchy-Vandermonde matrices, where it was noted that the entries of \eqref{BDAfac} can be expressed in terms of Schur polynomials.  
 Later,  a systematic research line was initiated in~\cite{DMR}, where explicit formulas for the initial minors of collocation matrices of arbitrary polynomial bases were derived in terms of Schur functions. 
 These results led to novel and practical criteria for determining the total positivity of polynomial bases. Subsequent applications of these findings were applied to Wronskian and Gramian matrices  in~\cite{DMR2,DMR3}, offering new insights and criteria for their total positivity.

In this paper, we extend the approach developed in~\cite{DMR,DMR2,DMR3} to the more general framework of bases of analytic functions. 
From the theoretical point of view, the main result is presented in Theorem \ref{theo1}, which establishes an expansion of the initial minors of the collocation matrix of a generic system of analytic functions in terms of Schur polynomials.  
From a practical perspective, Theorem \ref{theo1} also provides effective criteria for determining total positivity when applied to specific families of functions.

A particularly interesting case arises when considering any analytic function $f$ and the system formed by the functions  $f_i(x)= f(a_i x)$ for distinct parameters $a_i$, $i=1,\dots,n$. In this setting, the achieved expansions yield an infinite family of identities including, as a remarkable particular case, the celebrated Cauchy identity \eqref{eq:Cauchy}. 
The Cauchy identity establishes fundamental links between determinants, symmetric polynomials, and partition functions, 
 with deep connections to representation theory and algebraic combinatorics, inspiring various generalizations across mathematical structures \cite{Mac,Stanley,Bump,ortho,double,Hall,quantum}. Therefore, the detailed derivation of this type of identity, along with illustrative examples, deserves a dedicated section in this article.

The paper is organized as follows. In Section~\ref{sec:prelim}, we introduce notation and key concepts on total positivity and symmetric functions, which play a fundamental role in the paper. Section~\ref{sec:BD_factor} derives explicit formulas in terms of Schur polynomials for the initial minors \eqref{Aminors} of collocation matrices of systems of analytic functions. As a consequence, a sufficient condition to guarantee the total positivity property of these bases is derived. Furthermore,  
in Section~\ref{sec:cauchy_identities},   a family of Cauchy-type identities, extending the classical Cauchy identity, is derived.  Finally, in Section~\ref{sec:conclusions}, we summarize our main findings and suggest directions for future research.

\section{Total positivity and symmetric functions}\label{sec:prelim}
 In the sequel, we shall use the following matrix notation. Given   \( A \in \mathbb{R}^{n \times n} \), the submatrix formed by rows \( i_1, \ldots, i_r \) and columns \( j_1, \ldots, j_s \) will be denoted by  
\(
A[i_1, \ldots, i_r \mid j_1, \ldots, j_s].
\)  
For square submatrices, we shall use the shorthand notation   \(A[i_1, \ldots, i_r] := A[i_1, \ldots, i_r \mid i_1, \ldots, i_r]\).

 A symmetric function is a multivariate function that remains invariant under any permutation of its variables. Formally, a function \( f(x_1, \ldots, x_n) \) is symmetric if  
\[
f(x_{\sigma(1)}, \ldots, x_{\sigma(n)}) = f(x_1, \ldots, x_n),
\]
for any permutation \( \sigma \) of the indices \( \{1, \ldots, n\} \). Conversely, a function \( f \) is alternating if  
\[
f(x_{\sigma(1)}, \ldots, x_{\sigma(n)}) = \text{sgn}(\sigma) f(x_1, \ldots, x_n),
\]
where \(\text{sgn}(\sigma)\) denotes the signature of the permutation, taking values \( +1 \) or \( -1 \) depending on whether \( \sigma \) is even or odd.  

A matrix is called totally positive (TP) if all its minors are nonnegative and strictly totally positive (STP) if all its minors are strictly positive.

Given a system \( \mathbf{f} = (f_1, \ldots, f_n) \) of functions defined on \( I \subseteq \mathbb{R} \) and a sequence   \( \mathbf{x} = \{ x_1 , \ldots, x_n \} \) of parameters in $I$ satisfying \( x_1 < \cdots < x_n \), the collocation matrix of \( \mathbf{f} \) at \( \mathbf{x} \) is defined as  
\begin{equation}\label{eq:ColMF}
M_{\mathbf{f} ,\mathbf{x}}  := \left( f_j(x_i) \right)_{i,j=1}^n.
\end{equation}
The basis \( \mathbf{f} \) is called TP or STP if \( M_{\mathbf{f} ,\mathbf{x}}  \) is TP or STP, respectively, for all choices of \( \mathbf{x} \).  

When considering parameters \(\mathbf{x} = \{ x_1, \ldots, x_n \} \) and  the collocation matrix \( A = M_{\mathbf{f},\mathbf{x}} \) in \eqref{eq:ColMF}, the minors of $A$  are alternant (or antisymmetric) functions of the parameters.  Thus, although the collocation matrix \( M_{\mathbf{f}, \mathbf{x}} \) itself does not inherently possess any symmetry in its nodes, its initial minors do. This fundamental observation underpins the intriguing relationship between symmetric functions and TP bases, a connection that has motivated the present study, as well as the preceding works \cite{DMR,DMR2,DMR3}. 
 
Integer partitions form a key combinatorial structure in the theory of symmetric functions. An integer partition \( \lambda \) of size \( |\lambda| = k \) and  length \( \ell(\lambda) = n \)  is a tuple of positive integers $\lambda = (\lambda_1, \lambda_2, \ldots, \lambda_n)$,  such that 
\begin{equation}\label{eq:partition}
  \sum_{i=1}^{n} \lambda_i = k,\quad  \lambda_1 \geq \lambda_2 \geq \cdots \geq \lambda_n > 0.
\end{equation}
Each \( \lambda_i \) is called a  part of the partition. A partition can be visualized using a Young diagram, consisting of left-aligned box rows, where the \( i \)-th row contains \( \lambda_i \) boxes.  We will denote by $\Lambda $ the set of integer partitions and  $\Lambda_j $ the set of integer partitions such that $\lambda_1\le j$.  

 For a $n\times n$ matrix $A$, $1\le j\le n$, and an integer partition $\lambda = (\lambda_1, \ldots, \lambda_j)$ satisfying \eqref{eq:partition}, 
 we shall denote
\[
       A{ [i,\lambda]}:=  \det A [i-j+1,\dots, i\,|\lambda_j+1,\dots,\lambda_1+j], \quad j \le i \le n. 
   \]
 
Schur polynomials form a distinguished family of symmetric polynomials indexed by integer partitions. Note that, given   \( \lambda = (\lambda_1, \ldots, \lambda_n) \),  
\[
p(x_1, \dots, x_n) := \det \left( x_j^{\lambda_i + n - i} \right)_{i,j=1}^n
\]  
is an alternant polynomial and so, it is divisible by  the Vandermonde determinant  
\begin{equation}\label{eq:V}
V(x_1, \ldots, x_n) := \det \left( x_j^{n-i} \right)_{i,j=1}^n.
\end{equation}
The Schur polynomial  associated with \( \lambda \) is then defined  with the Jacobi's bi-alternant formula as:  
\begin{equation}\label{schurdef}
s_\lambda(x_1, \dots, x_n) := \frac{ p(x_1, \dots, x_n) }{ V(x_1, \ldots, x_n)}.
\end{equation}  
We recall the following well-known properties, which will be used throughout this article. We refer the interested reader to \cite{Mac} for further details.
\begin{itemize}
    \item[i)] If  $x_i>0$, $i=1,\ldots,n$, then 
    \(s_\lambda(x_1,\ldots,x_{n}) > 0\).
    \item[ii)] For any \( \lambda\in \Lambda \) such that $l(\lambda) > n$,  
    \( s_\lambda(x_1,\ldots,x_{n}) = 0.\)
    \item[iii)] Given  $\alpha\in\mathbb R$,
    \[
        s_\lambda(\alpha x_1,   \ldots, \alpha x_{n}) = \alpha^{|\lambda|} s_\lambda(x_1,\ldots,x_{n}).
    \]
\end{itemize}

By running over all partitions of size \( |\lambda| = k \), the corresponding Schur polynomials form a basis for the space of symmetric homogeneous polynomials of degree \( k \). More generally, when considering all partitions, they provide a basis for the space of symmetric functions.

\section{Initial minors of collocation matrices and symmetric functions} \label{sec:BD_factor}

Any alternant function  \( f  \)  can be expressed as  
\[ 
    f(x_1,\dots,x_n) = V(x_1, \ldots, x_n) \, g(x_1,\dots,x_n),
\] 
where \( g  \)  is a symmetric function and $V $ the Vandermonde determinant for the variables \((x_1,\ldots,x_n)\), as defined  in  \eqref{eq:V}.
This property extends to the initial minors of collocation matrices \( M_{\mathbf{f}, \mathbf{x}}   \), and we can write:  
\begin{eqnarray} \label{eq:gij}
 \det  M_{\mathbf{f}, \mathbf{x}}[i-j+1, \dots, i \mid 1, \dots, j] &=& V(x_{i-j+1}, \ldots, x_i) \, g_{i,j}(x_{i-j+1},\dots,x_i), \nonumber \\
  \det  M_{\mathbf{f}, \mathbf{x}}^t [i-j+1, \dots, i \mid 1, \dots, j] &=& V(x_1, \ldots, x_j) \, g_{j,i}(x_1,\dots,x_j),   
\end{eqnarray} 
for   $0 < j \le i \leq n$, with a suitable collection of \( n^2 \) symmetric functions \( g_{i,j} \) with \(i,j=1,\dots,n \), that can be expanded in the basis of Schur polynomials. 

The following result provides an explicit expansion \eqref{eq:gij} of the minors \eqref{Aminors} of collocation matrices of systems of analytic functions.  
As we are going to see, the coefficients  are fundamentally related to the minors of an infinite Wronskian matrix of   \( {\bf f} \), defined as  
\begin{equation}\label{extendedW}
       W_{{\mathbf f},  a}:=\big(f_j^{(i-1)}(a)\big)_{1 \le i < \infty; 1\le j \le n},\quad a\in I, 
\end{equation} 
where \(f ^{(i)}(a)\) denotes the $i$-th derivative of the function $f$ at the parameter $a$ in its domain. 
Without loss of generality, we set \( a = 0 \) and denote the resulting matrix as \( W_{\bf f} := W_{{\mathbf f},  0}  \).

\begin{theorem} \label{theo1}
Let  \( \mathbf{f} = (f_1, \ldots, f_n) \) be a system of analytic functions on \( I \subseteq \mathbb{R} \) and   \( \mathbf{x} = \{ x_1 , \ldots, x_n \} \) in \( I \) satisfying \( x_1 < \cdots < x_n \). Then 
\begin{equation}\label{MWT}
    \det M_{\bf f, \bf x}  {[i-j+1,\dots,i\,|\,1,\dots, j]}
 = V(x_{i-j+1},\dots,x_i)  \sum_{\lambda\in \Lambda}\frac{1}{C_\lambda}W^t_{{\bf f}}{[j,\lambda]}s_{\lambda}(x_{i-j+1},\dots,x_i),
\end{equation}
   and
   \begin{equation}\label{MWTt}
  \det M_{\bf f, \bf x}^t {[i-j+1,\dots,i\,|\,1,\dots, j]}
 = V(x_{1},\dots,x_j) \sum_{\lambda\in \Lambda}\frac{1}{C_\lambda}W^t_{{\bf f}}{[i,\lambda]}s_{\lambda}(x_{1},\dots,x_j),
\end{equation} 
where, $0 < j \leq i \leq n$, and,  for any $\lambda=(\lambda_1,\ldots,\lambda_j) \in \Lambda$,
\begin{equation}\label{eq:CLambda}
    C_\lambda:=\lambda_j!(\lambda_{j-1}+1)!\cdots (\lambda_1+j-1)!.
\end{equation}
     
\end{theorem}

\begin{proof}
The formula \eqref{MWT} is derived using the following equalities, which are justified below. The proof of \eqref{MWTt} follows  in an entirely analogous manner. 
    \begin{eqnarray*}
    &&\det M_{\bf f, \bf x}{[i-j+1,\dots,i\,|\,1,\dots, j]} =\det \left( \sum_{k=0}^\infty \frac{f_m^{(k)}(0)}{k!}x_l^{k}\right)_{i-j+1 \le l \le i; 1\le m \le j} \nonumber \\
    &\stackrel{^{(1)}}{=}&{  \sum_{r=1}^j \sum_{k_r=0}^\infty \det\left(  \frac{f_m^{(k_m)}(0)}{k_m!}x_l^{k_m}\right)_{l,m}
    =\sum_{r=1}^j \sum_{k_r=0}^\infty   \bigg[\prod_{m=1}^j \frac{f_m^{(k_m)}(0)}{k_m!}\bigg]\det\left( x_l^{k_m}\right)_{l,m} }\\
    &\stackrel{^{(2)}}{=}&   \sum_{k_1>\cdots> k_j=0}^\infty \sum_{\sigma\in S_j} \bigg[\prod_{m=1}^j \frac{f_{m}^{(k_{\sigma(m)})}(0)}{k_{\sigma(m)}!}\bigg]\det\Big(x_l^{k_{\sigma(m)}}\Big)_{l,m}  \nonumber\\
    &\stackrel{^{(3)}}{=}&\sum_{k_1>\cdots> k_j=0}^\infty \sum_{\sigma\in S_j} \bigg[\prod_{m=1}^j \frac{f_m^{(k_{\sigma(m)})}(0)}{   k_{m}!}\bigg]\text{sgn}(\sigma)\det\Big(x_l^{k_m}\Big)_{l,m}\nonumber\\
    &\stackrel{^{(4)}}{=}&\sum_{k_1>\cdots> k_j=0}^\infty  \bigg[\prod_{m=1}^j \frac{1}{k_m!}\bigg]\det\Big(f_r^{(k_s)}(0)\Big)_{1\leq s,r\leq j}\det\Big(x_l^{k_m}\Big)_{l,m}\nonumber \\
    &=&\sum_{k_1>\cdots> k_j=0}^\infty  \bigg[\prod_{m=1}^j \frac{1}{k_m!}\bigg]W_{\bf f}^t{[1,\dots,j|\,k_j+1,\dots,k_1{  +1}]}\det\Big(x_l^{k_m}\Big)_{l,m}\nonumber \\
    &\stackrel{^{(5)}}{=}&\sum_{\substack{ \lambda\in\Lambda\\
    l(\lambda)\leq j}}  \bigg[\prod_{m=1}^j \frac{1}{(\lambda_m+n-m)!}\bigg]W^t_{{\bf f}}{[j,\lambda]}\det\Big(x_l^{\lambda_m+j-m}\Big)_{l,m}\nonumber \\
    &\stackrel{^{(6)}}{=}& V( x_{i-j+1},\dots,x_i) \sum_{\lambda\in \Lambda}\frac{1}{C_\lambda}W^t_{{\bf f}}{[j,\lambda]}s_{\lambda}(x_{i-j+1},\dots,x_i),
\end{eqnarray*}
where, in the indicated identities, we have taken into account the following aspects.
\begin{enumerate}
    \item[(1)] 
Each $k_1,\dots,k_j$ takes the values of the degree of the terms in the McLaurin expansion of $f_1,\ldots, f_j$, respectively.  Then,  basic properties of determinants have been applied. 
    
    \item[(2)] If \(k_r = k_s\) for some \(r, s \in \{1, \dots, j\} \), we have  
    \[\det\Big(x_l^{k_m}\Big)_{l,m}=0.\]
 Therefore, the summatory over \(j\)-tuples \((k_1, \dots, k_j)\) reduces to a sum over tuples with distinct entries. Consequently, we can decompose the sum into two parts: the sum over ordered sequences \(k_1 > \cdots > k_j\) and the sum over all permutations of the elements of each sequence expressed as \((k_{\sigma(1)}, k_{\sigma(2)}, \dots, k_{\sigma(j)})\), where \(\sigma \in S_j\).

    \item[(3)] Be aware that, for $\sigma\in S_j$, we have 
    \begin{equation*}
      \prod_{m=1}^j \frac{1}{k_{\sigma(m)}!}= \prod_{m=1}^j \frac{1}{k_m!}\quad \text{ and }\quad \det\Big(x_l^{k_{\sigma(m)}}\Big)_{l,m}=\det\Big(x_l^{k_{m}}\Big)_{l,m}\text{sgn}(\sigma).
    \end{equation*}
    \item[(4)] Note that
    \begin{equation*}
        \sum_{\sigma\in S_j} \prod_{m=1}^j f_m^{(k_{\sigma(m)})}(0)\text{ sgn}(\sigma)=\det\Big(f_r^{(k_s)}(0)\Big)_{1\leq s,r\leq j}.
    \end{equation*}
    \item[(5)] 
After performing the substitution  
\begin{equation*}
    k_m = \lambda_m + j - m, \quad m = 1, \dots, j,
\end{equation*}  
each tuple \(k_1 > \cdots > k_j\) is transformed into a partition \(\lambda=(\lambda_1, \dots, \lambda_j)\) with at most \(j\) parts. It is straightforward to observe that summing over all sequences \(k_1 > \cdots > k_j\) is equivalent to summing over all partitions \(\lambda\in\Lambda\) with \(l(\lambda) \leq j\).  

\item[(6)] We have applied Jacobi's bialternant formula \eqref{schurdef}.  The restriction \(l(\lambda) \leq j\) is unnecessary in the summation since Schur functions of $j$ variables are zero whenever \(l(\lambda) > j\).  
\end{enumerate}
 \qed
\end{proof}

Note that if the functions of system ${\bf f}=(f_1,\ldots,f_n)$ had been expanded around $a\neq 0$,  formulas \eqref{MWT} and \eqref{MWTt}   would  be analogous, involving the Wronskian $W_{{\bf f},a} $ and the translation of the Schur functions $s_{\lambda}(x_{i-j+1}-a,\dots,x_i-a)$.

Let \( \mathbf{P}^{n}(I) \) denote the space of polynomials of degree at most \( n \), defined on \( I \subseteq \mathbb{R} \). 
The following result shows that, in the particular polynomial case, identities \eqref{MWT} and \eqref{MWTt} reduce to those found in \cite{DMR}.

\begin{corollary} Let \( \mathbf{p} = (p_{1}, \ldots, p_{n}) \) be a basis of \( \mathbf{P}^{n-1}(I) \) such that  
\( p_{i}(x) = \sum_{j=1}^{n} a_{i,j}\,x^{j-1}\),  for \(i = 1, \dots, n\), and define \( A := (a_{i,j})_{1 \leq i,j \leq n} \). Then, the following identities hold:
    \begin{equation}\label{initialP}
 \det M_{\bf p, \bf x} {[i-j+1,\dots,i\,|\,1,\dots, j]}= V( x_{i-j+1},\dots,x_i) \sum_{ \lambda\in\Lambda_{n-j}}A{[j,\lambda]}s_{\lambda}(x_{i-j+1},\dots,x_i),
\end{equation}
and
\begin{equation}\label{initialtransP}
   \det M_{\bf f, \bf x}^t {[i-j+1,\dots,i\,|\,1,\dots, j]}
 = V(x_{1},\dots,x_j) \sum_{ \lambda\in\Lambda_{n-j}}A{[j,\lambda]}s_{\lambda}(x_{i-j+1},\dots,x_i),
\end{equation}  
for $0 < j \leq i \leq n$.
\end{corollary}
\begin{proof}

Since $p_i^{(j)}(0)=0$ for $j\ge n$,    \( W^t_{\bf p} \) has nonzero entries only in the first \( n \) columns. Thus, the minors of \( W^t_{\bf p} \) involving columns \( j \geq n \) are zero and do not contribute to the sum in \eqref{MWT}. 
 Moreover,
\[
W_{\bf p}^t[1,\ldots,n] =  A  D, \quad D=\textrm{diag}(0!, 1!,\ldots, (n-1)!).
\]
Then, we obtain  
\begin{equation}\label{WPA}
W^t_{\bf p}{[l,\lambda]} =
\begin{cases}  
C_\lambda \,A{[l,\lambda]}, & \lambda_1 \leq n-j, \\  
0, & \lambda_1 > n-j,
\end{cases}  
\end{equation}  
with $C_\lambda$ as in \eqref{eq:CLambda}. Substituting \eqref{WPA} into \eqref{MWT}, we derive \eqref{initialP}. The formula \eqref{initialtransP} follows analogously.  
\qed
\end{proof}

As a consequence of Theorem \ref{theo1}, we can establish the following sufficient condition for the total positivity of any basis formed by analytic functions.

\begin{theorem}\label{theoWF}
Let  \( \mathbf{f} = (f_1, \ldots, f_n) \) be a system of analytic functions on \( I \subseteq \mathbb{R} \) such that the infinite Wronskian  \( W_{\bf f} \)  in \eqref{extendedW} is TP. Then 
    ${\bf f}$ is TP on $(0,\infty)$.
\end{theorem}
\begin{proof}
If \( W_{\bf f} \) is TP, we can guarantee that all its minors in \eqref{MWT} and \eqref{MWTt} are non-negative. 
Moreover, it is well-known that a Vandermonde matrix is TP for any \( 0 < x_1 < x_2 < \cdots < x_n \), and the Schur functions are positive for positive arguments. 
So, using Theorem  \ref{theo1}, we deduce that all the initial minors  \eqref{Aminors} of \(M_{\bf f, \bf x}\)  are nonnegative and can guarantee the total positivity of   \(M_{\bf f, \bf x}\) (see Section 4 of \cite{Gasca1}).\qed
\end{proof}

\begin{theorem}
Let \( \mathbf{f} = (f_1, \ldots, f_n) \) be a system of analytic functions on \( I \subseteq \mathbb{R} \). Then $W_{\bf f}$ is TP on $I$
      if and only if, for any nonnegative integer $k$, the systems \({\bf f}_k=(f_1^{(k)},\dots, f_n^{(k)})\)  are TP on $I$.  
\end{theorem}
\begin{proof}
   First, note that
   \[
   \Big(W_{{\bf f}_k}\Big)_{i,j}=\Big(W_{\bf f}\Big)_{i+k,j}.
   \] 
   Then, if $W_{\bf f}$ is TP so is $W_{{\bf f}_k}$. Considering Theorem \ref{theoWF}, we conclude that  ${\bf f}_k$ is TP.  
    For the inverse implication, use Theorem 2 of \cite{DMR2}, which assures that ${\bf f}_k$ being TP for all nonnegative $k$ is a sufficient condition for $W_{\bf f}$ to be TP. \qed
\end{proof}

\section{Explicit formulation of Cauchy-type identities}\label{sec:cauchy_identities}
 
The well-known   Cauchy identity 
\begin{equation} \label{eq:Cauchy}
\prod_{i,j=1}^n \frac{1}{1 - x_i y_j} = \sum_{\lambda\in\Lambda} s_\lambda(x_1, \ldots, x_n) s_\lambda(y_1,\ldots, y_n)
\end{equation} 
reveals deep connections between seemingly disparate entities, such as determinants, symmetric polynomials, and partition functions. Rigorous algebraic proofs
of the identity can be found in classical references such as \cite{Mac} and \cite{Stanley}.  Extensions exist for characters of Orthogonal and Symplectic groups \cite{ortho}, double-Schur polynomials \cite{double}, Hall-Littlewood polynomials \cite{Hall}, and
$q$-deformations \cite{quantum}, among others. Each of these adaptations reveals new layers of insight into the interplay between algebraic structures.

In this section,  the determinant expansions in Theorem \ref{theo1} considered for a particular family of collocation matrices are analyzed to derive identities, with the classical Cauchy formula \eqref{eq:Cauchy} emerging as a special case. 

Let us suppose that  $f$ is a function infinitely differentiable at \(a=0\) and  $R> 0$ is the radius of convergence of its MacLaurin series. 
  For a positive increasing sequence   \( a_1,  \ldots, a_n \),  we can  define  on  $I=(-R/a_n,R/a_n)$ a system \( {\bf f} = (f_1, \dots, f_n) \) with  
\begin{equation}\label{defg}
  f_i(x) := f(a_i x),\quad i = 1, \dots, n.
\end{equation}

The systems defined as in \eqref{defg} form a broad family and, as we will show in this section, simpler criteria for total positivity apply to them. 
The following result illustrates that the expansion of their initial minors leads to Cauchy-type identities.  
 
\begin{theorem}\label{theo}
 Let \( {\bf f} = (f_1,  \dots, f_n) \) be the system defined as in \eqref{defg}. Given ${\bf x}=(x_1,\dots, x_n)$ on $I$,  we have:
  \begin{equation}\label{MFG}
  \frac{ \det M_{\bf f, \bf x}  {[i-j+1,\dots,i\,|\,1,\dots, j]} }{ V(a_{1},\dots,{
  a_j})V(x_{i-j+1},\dots,x_i)} =  \sum_{\lambda\in\Lambda}\frac{F_\lambda}{C_\lambda}s_{\lambda}(a_{1},\dots,a_j)s_{\lambda}(x_{i-j+1},\dots,x_i),
\end{equation}
for  $0 < j \leq i \leq n$, where $C_{\lambda}$ is defined in \eqref{eq:CLambda}, $V$ denotes the Vandermonde determinant  defined in  \eqref{eq:V}, and
\begin{equation}  \label{eq:FLambda}
    F_\lambda:=\prod_{l=1}^j f^{(\lambda_l+j-l)}(0). 
\end{equation}
Moreover, 
\begin{equation}\label{MFtG}
  \frac{     \det M_{\bf f, \bf x}^t {[i-j+1,\dots,i\,|\,1,\dots, j]}}{V(a_{i-j+1},\dots,a_i)V(x_{1},\dots,x_j) }= \sum_{\lambda\in\Lambda}\frac{F_\lambda}{C_\lambda}s_{\lambda}(a_{i-j+1},\dots,a_i)s_{\lambda}(x_{1},\dots,x_j),  
\end{equation}  
for $0 < j \leq i \leq n$.
\end{theorem}
\begin{proof}
The elements of the infinity Wronskian \( W^t_{\bf f} \) defined in  \eqref{extendedW} satisfy
\[   
\Big(W^t_{\bf f}\Big)_{l,\,k+1} = f_l^{(k)}(0) = a_l^k\,f^{(k)}(0).  
\]
Let \( A_j = (a_{r,s})_{r,s=1}^j \), with \( a_{r,s} = a^{\lambda_{j-s+1} + s - 1}_{i-j+r} \). Then, we have  
\begin{equation}\label{WFG}
    W_{\bf f}^t{[i,\lambda]} = f^{(\lambda_j)}(0)\cdots f^{(\lambda_1+j-1)}(0) \det A_j  
    = F_\lambda V( a_{i-j+1}, \dots, a_i) s_{\lambda}(a_{i-j+1}, \dots, a_i).
\end{equation}  
Substituting \eqref{WFG} into \eqref{MWTt} and \eqref{MWT}, formulae \eqref{MFG} and \eqref{MFtG} are derived.  \qed
\end{proof}

Now, from formula \eqref{MFG}, we derive the following condition that guarantees the total positivity of the systems  defined by \eqref{defg}.

\begin{theorem}\label{TheoremTP}
Let \( {\bf f} = (f_1,  \dots, f_n) \) be the system defined as in \eqref{defg} for a function $f$ such that $f(0)\geq 0$  and $f^{(k)}(0)\geq 0$ for $k\in \mathbb{N}$.  Then ${\bf f}$ is TP on the interval $(0,R/a_n)$. 
\end{theorem}
\begin{proof}
For \( 0 \leq a_1 < \cdots < a_n \) and \( 0 \leq x_1 < \cdots < x_n \), the Vandermonde determinant for any ordered subsequence of \({\bf x} = (x_1, \ldots, x_n)\) or \({\bf a} = (a_1, \ldots, a_n)\) is positive.  
Moreover, for any partition \(\lambda\), both the constants \(C_\lambda\) and the Schur functions \(s_\lambda\) are positive. Consequently,  if \(f^{(k)}(0) \geq 0\) for all \(k \in \mathbb{N}\), it follows from Theorem \ref{theo} that \(F_\lambda \geq 0\). 
This implies that all the initial minors  \eqref{Aminors} of \(M_{\bf f, \bf x}\) and \(M^t_{\bf f, \bf x}\) are nonnegative, thereby ensuring the total positivity of the collocation matrix \(M_{\bf f, \bf x}\) (see Section 4 of \cite{Gasca1}).    \qed
\end{proof}

As a direct consequence of Theorem \ref{theo}, Cauchy-type identities are derived.

\begin{corollary} \label{Theorem}
  Let  \( {\bf f} = (f_1,  \dots, f_n) \) be the system defined  as in \eqref{defg} and    \(M_{\bf f,  \bf x} \) its collocation matrix 
at the increasing sequence ${\bf x}=(x_1,\dots, x_n)$. Then
\begin{equation}\label{Cor}
 \sum_{\lambda\in \Lambda  }\frac{F_\lambda}{C_\lambda}s_{\lambda}(a_1,\dots,a_n)s_{\lambda}(x_{1},\dots,x_n)=\frac{\det  M_{\bf f,  \bf x} }{ V(a_{1},\dots,a_n) V(x_{1},\dots,x_n)},
\end{equation}
where $C_{\lambda}$    is defined in \eqref{eq:CLambda}, $F_{\lambda}$    is defined in \eqref{eq:FLambda} and $V$ denotes the Vandermonde determinant  in  \eqref{eq:V}, 
\end{corollary}

Observe that any permutation  of the entries \((a_{\sigma(1)}, \dots, a_{\sigma(n)})\) or \((x_{\sigma(1)}, \dots, x_{\sigma(n)})\) leaves both sides of \eqref{Cor} invariant.  Furthermore, the symmetry under the exchange \((a_1, \dots, a_n) \leftrightarrow (x_1, \dots, x_n)\) on the left-hand side is reflected on the right-hand side of  \eqref{Cor}. This follows because such an exchange transposes the collocation matrix \( M_{\bf f,  \bf x}\), leaving its determinant unchanged.

Also, note that a function \(f\) uniquely determines both the left-hand side of \eqref{Cor}, through the coefficients \(F_\lambda\), and the right-hand side, via the collocation matrix \(M_{\bf f,  \bf x}\). Consequently, each analytic function \(f\) gives rise to a distinct Cauchy-type identity. This result highlights the versatility of Corollary \ref{Theorem}, which will be further demonstrated in the following cases. Further illustrative examples can be found in \cite{DM}.

\begin{exmp}
For a single variable $x_1$ and $a_1=1$, formula \eqref{Cor} reduces to the Maclaurin expansion of \(f(x)\). Notice that, in this case, only the partitions with one part, $l(\lambda)=1$, contribute to the sum. On the other hand, for $|\lambda|=k$, we have
\begin{equation*}
    C_\lambda = k!, \quad F_\lambda = f^{(k)}(0), \quad \det M_{\bf f,  \bf x} = f(x_1),
\end{equation*}
thus recovering the Maclaurin series of $f$:
\[
\sum_{k=0}^\infty \frac{ f^{(k)} (0) }{k!}x_1^k=f(x_1).
\]
\end{exmp}

\begin{exmp}
When considering $f(x)=\frac{1}{1-x}$, we have $f(0)=1$,  $f^{(k)}(x)= k! (1-x)^{-k-1}$, and thus:
\[
F_\lambda =\prod_{l=1}^n  f^{(\lambda_l+n-l)}(0)=  (\lambda_1+n-1)!\cdots  (\lambda_{n-1}+1)! \, \lambda_n!=C_\lambda.\] 
On the other hand, using a well-known variant of the classical Cauchy determinant identity:
\begin{equation*}
    \det\left( \frac{1}{x_i - y_j}\right)_{i,j=1}^n = \frac{  \prod_{1\le i < j \le n } (x_j-x_i)   (y_i-y_j) } { \prod_{1\le i,j \le n }^n(x_i - y_j)},
\end{equation*}
we derive 
\begin{eqnarray*}
     \det M_{\bf f,  \bf x} &=& \det \left( \frac{1}{1 - a_j x_i} \right)_{i,j=1}^n =
     \det \left( \frac{-1/a_j}{x_i -\frac{1}{ a_j} }\right)_{i,j=1}^n = 
 \frac{ (-1)^n }{\prod_{j=1}^na_j}   
 \det \left( \frac{ 1 }{x_i -\frac{1}{ a_j} }\right)_{i,j=1}^n \\
&=& \frac{ (-1)^n }{\prod_{j=1}^na_j}      \frac{  \prod_{1\le i < j \le n } \left(\frac{1}{a_i}-\frac{1}{a_j} \right)(x_j-x_i)   }{ \prod_{1\le i , j \le n} \left(x_i -\frac{1}{ a_j}\right)} \\
&=& 
        (-1)^n\frac{\prod_{i=1}^na_i^{n-1}}{  \prod_{1\le i < j \le n}a_i a_j}  \frac{   \prod_{1\le i< j \le n} (a_j-a_i)(x_j-x_i)}{\prod_{1\le i , j \le n} (a_jx_i -1)}. 
 \end{eqnarray*}
Since 
\[ 
\prod_{1\le i < j \le n } a_i a_j = \prod_{i=1}^n a_i^{n-1},
\]
we can write:
\begin{eqnarray}
    \det M_{\bf f,  \bf x}&=&   (-1)^n    \frac{    \prod_{1\le i < j \le n }  (a_j-a_i )(x_j-x_i)    }{  \prod_{1\le i , j \le n} \left(a_j x_i  - 1 \right)}  = 
 (-1)^{n(1-n)}    \frac{    \prod_{1\le i < j \le n }     (a_j-a_i )(x_j-x_i) }{  \prod_{1\le i , j \le n} \left(1- a_j x_i   \right)} \nonumber \\
 &=&   \frac{    \prod_{1\le i < j \le n }    (a_j-a_i )(x_j-x_i)  }{  \prod_{1\le i , j \le n} \left(1- a_j x_i   \right)}=\frac{  V( a_1,\dots,a_n) V( x_1,\dots,x_n)   }{\prod_{1\le i , j \le n}(1 - a_j x_i)}.\label{eq:detC0}
\end{eqnarray}
This recovers the Cauchy identity
\[
     \sum_{\lambda\in\Lambda} s_\lambda(a_1,\dots,a_n)s_\lambda(x_1,\dots,x_n) = \frac{\det \left( \frac{1}{1 - a_j x_i} \right)_{i,j=1}^n }{ V( a_1,\dots,a_n) V( x_1,\dots,x_n)  }=  \prod_{1\le i,j\le n}\frac{1}{1 - a_j x_i}.
\]
 From Theorem \ref{TheoremTP}, given $0<a_1<\cdots<a_n$, we can guarantee that the system of rational functions 
$$
{\mathbf f}=\left( \frac{1}{1-a_1x}, \frac{1}{1-a_2x},\ldots , \frac{1}{1-a_nx}\right)
$$
is totally positive on the interval $(0,1/a_n)$.
\end{exmp}
\begin{exmp}
     For   
    \(f(x)=\frac{1}{1-x^2}\),  we have $f(0)=0$,  
    \[
    f^{(k)}(x)= k!  \left( (1-x)^{-k-1} + (-1)^k (1+x)^{-k-1} \right)/2,
    \] 
    and $f^{(k)}(0)=0$ if $k$ is an odd integer.  Thus:
         \begin{equation*}
                 F_\lambda= \prod_{l=1}^n  f^{(\lambda_l+n-l)}(0)=\begin{cases}  (\lambda_1+n-1)!\cdots  (\lambda_{n-1}+1)! \, \lambda_n!=C_\lambda,  & \textrm{ if }  \lambda \in \Lambda^n_{even}, \\ 0, & \textrm{otherwise},  \end{cases}    
    \end{equation*}
where     \( \Lambda^n_{even} := \{ \lambda \in  \Lambda \mid \lambda_l + n - l \textrm{ is even for all } l = 1, \dots, n\}\).
Now, using \eqref{eq:detC0}, we derive:
  \[ 
           M_{\bf f,  \bf x}=\left( \frac{1}{1-a^2_jx^2_i}\right)_{i,j=1}^n 
           = \frac{    \prod_{1\le i < j \le n } (x_j -x_i )  (x_j+x_i)  (a_j -a_i ) (a_j+a_i  ) }{  \prod_{1\le i , j \le n} \left(1- a_j  x_i  \right) \left(1+ a_j  x_i \right) },
  \]
and,  from \eqref{Cor},   obtain the following Cauchy-type identity:
    \begin{equation}\label{minusx2}
      \sum_{\lambda \in \Lambda^n_{even}}  s_\lambda(a_1,\dots,a_n)s_\lambda(x_1,\dots,x_n)=\frac{ \prod_{1\le i < j \le n }(a_i+a_j)(x_i+x_j)}{\prod_{1\le i , j \le n }(1-a_jx_i) (1+a_jx_i)}.
    \end{equation}
On the other hand, from Theorem \ref{TheoremTP}, we can guarantee that,  for $0<a_1<\cdots<a_n$,  the system of rational functions 
$$
{\mathbf b}=\left( \frac{1}{1-a_1^2x^2}, \frac{1}{1-a_2^2x^2},\ldots , \frac{1}{1-a_n^2x^2}\right)
$$
 is totally positive on the interval $(0,1/a_n)$.

Furthermore, from \eqref{eq:detC0}, we can write 
\begin{eqnarray*} \label{eq:colocacionEj4}
   &&\det \left( \frac{1}{1+a^2_jx^2_i}\right)_{i,j=1}^n=  
      \frac{   \prod_{1\le i< j \le n} (x_j-x_i)(x_j+x_i) (a_j+a_i)(a_i-a_j)}{\prod_{1\le i , j \le n} (1+a_j^2x_i^2)}\\ 
      &=& (-1)^{n(n-1)/2}   
      \frac{   \prod_{1\le i< j \le n} (x_j+x_i) (a_j+a_i) V( x_1,\dots,x_n)  V( a_1,\dots,a_n)}{\prod_{1\le i , j \le n} (1+a_j^2x_i^2)}.
 \end{eqnarray*}

Let  \(\mathbf{i} \) denote the imaginary unit, such that \(\mathbf{i}^2= -1\), and 
consider the transformation \(a_k\longrightarrow \mathbf{i} a_k\), for \(k=1,\dots, n\). By properties of the determinant, we have 
\[
 V( \mathbf{i} a_1,\dots, \mathbf{i}a_n) =\mathbf{i} ^{n(n-1)/2}  V( a_1,\dots,a_n),   
\]
and, using properties of Schur polynomials, 
\[
s_\lambda(\mathbf{i} a_1,\dots,\mathbf{i} a_n)= \mathbf{i} ^{|\lambda|} s_\lambda(a_1,\dots, a_n).
\]
Taking into account the previous identities in  \eqref{minusx2}, we derive 
 \begin{eqnarray*}
      \sum_{\lambda \in \Lambda^n_{even}}  s_\lambda( \mathbf{i}  a_1,\dots,  \mathbf{i} a_n)s_\lambda(x_1,\dots,x_n)&=&  
      \frac{  \det\left( \frac{1}{1+a^2_jx^2_i}\right)_{i,j=1}^n}{  
      V( x_1,\dots,x_n)  V( \mathbf{i}a_1,\dots,\mathbf{i}a_n)}\\ 
      &=& \mathbf{i}^{\frac{n(n-1)}{2} } \frac{  \prod_{1\le i< j \le n} (x_j+x_i) (a_j+a_i) } { \prod_{1\le i , j \le n} (1+a_j^2x_i^2)}. 
  \end{eqnarray*}
Equivalently, we can write
 \[
      \sum_{\lambda \in \Lambda^n_{even}}  \mathbf{i} ^{|\lambda|}  s_\lambda( a_1,\dots, a_n)s_\lambda(x_1,\dots,x_n)=  
     (-1)^{\frac{n(n-1)}{4} } \frac{  \prod_{1\le i< j \le n} (x_j+x_i) (a_j+a_i) } { \prod_{1\le i , j \le n} (1+a_j^2x_i^2)},
 \]
which implies the following Cauchy-type identity
\begin{equation}\label{plusx2}
      \sum_{\lambda \in \Lambda^n_{even}}  (-1)^{\frac{2 |\lambda|-n(n-1)}{4}}s_\lambda(a_1,\dots,a_n)s_\lambda(x_1,\dots,x_n)  =\frac{ \prod_{1\le i < j\le n}(a_i+a_j)(x_i+x_j)}{\prod_{1\le i,j\le n}(1+a^2_ix^2_j)}.
\end{equation}
\end{exmp}
 Note that \eqref{plusx2} could also have been derived from Theorem \ref{theo} by taking  $f(x)=1/(1+x^2)$.

\section{Conclusions and further research}\label{sec:conclusions}
In this paper, we have studied the bidiagonal factorization of the collocation matrices of analytic bases using symmetric functions. Specifically, we have obtained the expansion of the initial minors of their collocation matrices on the Schur polynomial basis of symmetric functions. The structure of the expansions has enabled us to establish sufficient conditions for the total positivity of generic systems of analytic functions. Besides, the expansions have been found to lead to generalizations of the Cauchy identity for certain families of functions.  

This work provides a solid theoretical framework for various applications. 
For instance, the general expansions derived in Theorem \ref{theo1} can be applied to the case of Wronskian and Gram matrices, where the expansion of their initial minors was previously studied for the
polynomial case in   \cite{DMR2} and \cite{DMR3}, respectively. This will result in more general and unified expressions
for their initial minors. 
Furthermore, exploring the families defined by \eqref{defg} in this context
may lead to simplified and more insightful formulas.
 
An interesting direction for future research is to explore the deeper theoretical meaning
of the Cauchy-type identities established in this work.   Among the many interpretations of the classical Cauchy identity, one of the most illuminating arises from representation theory, where it reflects properties of the characters of the General Linear group (see, for instance,
  \cite{Bump}, Chapter
38), which makes it a fundamental piece in algebraic combinatorics. The Cauchy-type identities found in this paper likely admit an analogous interpretation. 

Finally, working with totally positive bases offers a significant advantage in numerical
computations, particularly in achieving high relative accuracy (HRA). The analytic bases
studied in this paper, especially the families defined by \eqref{defg}, provide a promising framework
for precision numerical calculations using appropriate algorithms. This suggests a strong
potential for further development in both theoretical and computational directions.

 
\bigskip

\section*{Statements and Declarations}

\noindent{\bf Competing interest} The authors declare no competing interests.\\
 
\noindent{\bf Funding} This work was partially supported by Spanish research grants PID2022-138569NB-I00 (MCI/AEI) and RED2022-134176-T (MCI/AEI) and by Gobierno de Aragón (E41$\_$23R, S60$\_$23R). \\

\noindent{\bf Data Availability}  Enquiries about data availability should be directed to the authors.

\end{document}